\documentclass{article}

\usepackage{geometry}
\usepackage[]{amsthm,amsmath,amssymb,amsfonts,latexsym,hyperref}

\newcounter{dummy} \numberwithin{dummy}{section}
\newtheorem{theorem}[dummy]{Theorem}
\newtheorem{proposition}[dummy]{Proposition}
\newtheorem{lemma}[dummy]{Lemma}
\newtheorem{corollary}[dummy]{Corollary}

\newtheorem{remark}[dummy]{Remark}
\let\tempremark\remark
\renewcommand{\remark}{\tempremark\normalfont}

\def\R{\mathbb{R}}

\def\Z{\mathbb{Z}}

\def\P{\mathbb{P}}
\def\E{\mathbb{E}}
\def\1{\mathbf{1}}

\def\({\left(}
\def\){\right)}
\def\mmid{\,\middle|\,}
\def\longeqpad{\hphantom{xxxxxxxxx}}

\def\taucov{\tau_{\text{cov}}}
\def\tcov{t_{\text{cov}}}
\def\tauhit{\tau_{\text{hit}}}
\def\thit{t_{\text{hit}}}

\def\Reff{R_{\text{eff}}}

\def\ctot{c_{\text{tot}}}
\def\L{\mathcal{L}}

\newcommand{\laweq}{\stackrel{\text{law}}{=}}

\begin{document}

\title{Exponential concentration of cover times}
\author{Alex Zhai \\ Stanford University}
\maketitle

\begin{abstract}
  We prove an exponential concentration bound for cover times of
  general graphs in terms of the Gaussian free field, extending the
  work of Ding, Lee, and Peres \cite{DLP12} and Ding \cite{D14}. The
  estimate is asymptotically sharp as the ratio of hitting time to
  cover time goes to zero.

  The bounds are obtained by showing a stochastic domination in the
  generalized second Ray-Knight theorem, which was shown to imply
  exponential concentration of cover times by Ding in \cite{D14}. This
  stochastic domination result appeared earlier in a preprint of Lupu
  \cite{L14}, but the connection to cover times was not mentioned.
\end{abstract}

\section{Introduction}

Let $G = (V, E)$ be an undirected graph, possibly with self-loops and
multiple edges. For the continuous time simple random walk on $G$
started at a given vertex $v_0 \in V$, define $\taucov$ to be the
first time that all the vertices in $V$ have been visited at least
once. This quantity, known as the \emph{cover time}, is of fundamental
interest in the study of random walks.

Another fundamental object in the study of random walks on graphs is
the \emph{Gaussian free field} (GFF). For purposes of stating our main
result, let us define the GFF $\{ \eta_x \}_{x \in V}$ on $G$ with
$\eta_{v_0} = 0$ to be the Gaussian process given by covariances
$\E(\eta_x - \eta_y)^2 = \Reff(x, y)$, where $\Reff$ denotes effective
resistance. More background is given in Section \ref{sec:prelim}.

Our main result is the following concentration bound on the cover time
in terms of the Gaussian free field.

\begin{theorem} \label{thm:cover-concentration}
  Let $G = (V, E)$ be an undirected graph with a specified initial
  vertex $v_0 \in V$. Let $\{ \eta_x \}_{x \in V}$ be the Gaussian
  free field on $G$ with $\eta_{v_0} = 0$. Define the quantities
  \[ M = \E \max_{x \in V} \eta_x , \;\; R = \max_{x, y \in V} \Reff(x, y) = \max_{x, y \in V} \E \( \eta_x - \eta_y \)^2. \]
  Then, there are universal constants $c$ and $C$ such that for the
  continuous time random walk started at $v_0$, we have
  \[ \P\( \Big| \taucov - |E|M^2 \Big| \ge |E| (\sqrt{\lambda R} \cdot M + \lambda R) \) \le Ce^{-c \lambda} \]
  for any $\lambda \ge C$.
\end{theorem}
\begin{remark}
  Our result is most easily stated for a continuous time random walk,
  i.e. a random walk having the same jump probabilities as a simple
  random walk, but whose times between jumps are i.i.d. unit
  exponentials. However, note that if a continuous time random walk
  has run for time $t$, then the number of jumps it has made has
  Poisson distribution with mean $t$, which exhibits Gaussian
  concentration with fluctuations of order $\sqrt{t}$. Thus, Theorem
  \ref{thm:cover-concentration} can be easily translated into a
  similar bound for discrete random walks.
\end{remark}
\begin{remark}
  Note that the definition of $M$ is given in terms of a starting
  vertex $v_0$, but it does not depend on $v_0$. Indeed, let $v'_0$ be
  another starting vertex. Then, $\eta' = \{ \eta_x - \eta_{v'_0}
  \}_{x \in V}$ has the law of a GFF with $\eta'_{v'_0} = 0$, and
  \[ \E \max_{x \in V} \eta'_x = \E \max_{x \in V} \eta_x. \]
\end{remark}
\begin{remark}
  We actually show Theorem \ref{thm:cover-concentration} in the
  slightly more general setting of electrical networks, which are
  introduced in Section \ref{sec:prelim}.
\end{remark}

We prove Theorem \ref{thm:cover-concentration} following the approach
first appearing in a paper of Ding, Lee, and Peres \cite{DLP12} and
later refined by Ding \cite{D14}. Indeed, Ding observed that Theorem
\ref{thm:cover-concentration} is implied by a certain stochastic
domination; in \cite{D14}, the domination was proved for trees, but
the general case was left as conjecture (\cite{D14}, Question 5.2). We
establish Theorem \ref{thm:cover-concentration} by proving the
stochastic domination for general graphs.

In relation to these previous works, Theorem
\ref{thm:cover-concentration} extends Theorem 1.2 of \cite{D14}, which
gave the same concentration bound for trees. It also sharpens Theorem
1.1 of \cite{DLP12}, where the equivalence of cover times and $|E|M^2$
(in the notation of Theorem \ref{thm:cover-concentration}) was proven
up to a universal multiplicative constant. By ``sharpen'', we mean
that we are able to remove the constant factor under the assumption
$\sqrt{R} \ll M$. We mention that this was done already for
bounded-degree graphs in Theorem 1.1 of \cite{D14}, albeit without
exponential tail bounds.

The condition $\sqrt{R} \ll M$ is a relatively mild one. Indeed,
define $\tauhit(x, y)$ to be the time it takes for a random walk
started at $x$ to hit $y$, and define
\[ \thit = \max_{x, y \in V} \E \tauhit(x, y), \;\; \tcov = \max_{x \in V} \E_x \taucov, \]
where in the definition of $\tcov$, $\E_x$ denotes the expectation for
the random walk started at $x$. The well-known commute time identity
(\cite{LPW09}, Proposition 10.6) states that
\[ \E \tauhit(x, y) + \E \tauhit(y, x) = 2 |E| \cdot \Reff(x, y). \]
It follows that
\[ \thit \ge |E| \cdot R. \]
On the other hand, it was shown in \cite{DLP12} that $|E| \cdot M^2$
is within a constant of $\tcov$. It follows that for some constant
$C$,
\[ \frac{R}{M^2} \le C \cdot \frac{\thit}{\tcov}, \]
so $\sqrt{R} \ll M$ holds whenever $\thit \ll \tcov$. We obtain the
following corollary.

\begin{corollary} \label{cor:cover-concentration}
  Let $G = (V, E)$, $v_0$, $\eta$, $M$, and $R$ be as in Theorem
  \ref{thm:cover-concentration}. Then,
  \[ \( 1 - C \sqrt{\frac{\thit}{\tcov}} \) \cdot |E| \cdot M^2 \;\; \le \;\tcov\; \le \;\; \( 1 + C \sqrt{\frac{\thit}{\tcov}} \) \cdot |E| \cdot M^2, \]
  for a universal constant $C$.
\end{corollary}
\begin{remark}
  There is a deterministic polynomial-time approximation scheme (PTAS)
  due to Meka \cite{M12} for computing the supremum of a Gaussian
  process. Applying this to the quantity $M$ gives a PTAS for $\tcov$
  when $\thit \ll \tcov$.
\end{remark}

Conversely, it was shown by Aldous \cite{A91} that if $\thit$ is of
the same order as $\tcov$, then the cover time cannot be concentrated
about its expectation (see the introduction of \cite{D14} for a more
detailed discussion).

The main tool in estimating cover times employed in \cite{DLP12} and
\cite{D14} is the generalized second Ray-Knight theorem, which is an
identity in law relating the Gaussian free field to the time spent at
each vertex by a continuous time random walk. In fact, the upper bound
on $\tcov$ in Corollary \ref{cor:cover-concentration} was previously
established as Theorem 1.4 of \cite{DLP12} (the same argument also
proves the corresponding upper tail estimate in Theorem
\ref{thm:cover-concentration}).

In \cite{D14}, the matching lower bound was reduced to proving a
certain stochastic domination in the generalized second Ray-Knight
theorem. There, the stochastic domination was proven only for trees
(\cite{D14}, Theorem 2.3), but it was asked whether the same holds for
general graphs (\cite{D14}, Question 5.2).

Indeed, in Section \ref{sec:domination} we prove Theorem
\ref{thm:domination}, which generalizes Theorem 2.3 of \cite{D14} to
arbitrary graphs. This is accomplished by viewing the random walk as
Brownian motion on a metric graph. After writing up an early draft of
the proof, it was pointed out to us that this idea appeared previously
in a recent preprint of Lupu \cite{L14} to prove essentially the same
result (\cite{L14}, Theorem 3). In that context, the idea was mainly
used to study the percolation of loop clusters (\cite{L14}, Theorems 1
and 2; see also subsequent work by Sznitman \cite{S14}). However, the
application to cover times was not mentioned.

Even though Theorem \ref{thm:domination} uses the same ideas as
Theorem 3 of \cite{L14}, we include a proof in order to establish the
result in the language of our specific application. Additionally, our
exposition is intended to be more accessible to audiences interested
in cover times of random walks.

\subsection{Related work on cover times}

Cover times have been studied in many papers over the last few
decades. We highlight several of them below; see also \S 1.1 of
\cite{DLP12} for further background.

We first mention some results relating cover times and hitting
times. Clearly, $\tcov \ge \thit$. A classical result of Matthews
\cite{M88} is that on a graph of $n$ vertices, $\tcov \le \thit (1 +
\log n)$. This was proved by a clever argument analogous to the
analysis of the coupon collector's problem. Matthews also gave an
expression for a lower bound, which was later shown by Kahn, Kim,
Lovasz, and Vu \cite{KKLV02} to approximate the cover time to within
$(\log \log n)^2$.

In \cite{A91}, Aldous analyzed a generalization of the coupon
collector's problem. As a consequence, he showed that $\taucov$ is
concentrated around its expectation with high probability as
$\frac{\thit}{\tcov} \rightarrow 0$. More precisely, for any $\epsilon
> 0$, there is a small enough $\delta$ so that
\[ \P \( |\taucov - \tcov| \le \epsilon \tcov \) \ge 1 - \epsilon \]
whenever $\frac{\thit}{\tcov} < \delta$. This shows qualitatively the
concentration of cover times.

On the other hand, cover times have also been estimated for many
specific classes of graphs, including regular graphs \cite{JLNS89},
lattices \cite{Z92}, and bounded degree planar graphs \cite{JS00}, to
name a few. Precise asymptotics are known for the two-dimensional
discrete torus \cite{DPRZ04} and regular trees \cite{A91b}.

More recently, a breakthrough was made by Ding, Lee, and Peres
\cite{DLP12} whereby the cover time was given (up to a constant
factor) in terms of the Gaussian free field. Their result gives in
some sense a quantitative estimate of the cover time that works for
any graph. As touched upon earlier, Ding \cite{D14} later removed the
constant factor for trees and bounded degree graphs. We complete the
picture by extending this to general graphs.

\subsection{Outline}

The remaining sections are organized as follows. In Section
\ref{sec:prelim}, we establish notation and provide a brief review of
electrical networks, local times, Gaussian free fields, and the
generalized second Ray-Knight theorem. The notation mostly follows
\cite{D14}. Section \ref{sec:domination} is devoted to proving the
aforementioned stochastic domination in the form of Theorem
\ref{thm:domination}. This is very similar to Theorem 3 of \cite{L14};
nevertheless, we include a proof in the notation of our setting. In
Section \ref{sec:cover-times}, we apply Theorem \ref{thm:domination}
to cover times to obtain Theorem \ref{thm:cover-concentration}. The
final section contains acknowledgements.

\section{Definitions and preliminaries} \label{sec:prelim}

An \emph{electrical network} $G$ is a finite, undirected graph $(V,
E)$, allowing self-loops, together with positive weights on the edges
called \emph{conductances}. We use either $c_{xy}$ or $c_{yx}$ to
denote the conductance of an edge $(x, y)$, and for vertices $x, y \in
V$ that do not share an edge, we define $c_{xy} = 0$. It is convenient
to define the quantity $c_x = \sum_{y \in V} c_{xy}$, which we refer
to as the \emph{total conductance} at $x$.

The name ``electrical network'' comes from the fact that $G$ can be
used to model an electric circuit, where each edge $(x, y)$
corresponds to placing a resistor with resistance $\frac{1}{c_{xy}}$
between vertices $x$ and $y$. For any $x, y \in G$, we can define the
\emph{effective resistance} $\Reff(x, y)$ between $x$ and $y$ to be
the physical resistance when a voltage is applied between $x$ and
$y$. Mathematically, this quantity can be defined as a certain minimum
energy (see Chapter 9 of \cite{LPW09} for more background on effective
resistance and electrical networks).

There is a canonical \emph{discrete time random walk} on an electrical
network defined by taking the transition probability from $x$ to $y$
to be $\frac{c_{xy}}{c_x}$. In the case where the non-zero
conductances are all equal, this reduces to the simple random walk on
the underlying graph.

We will also want to consider the \emph{continuous time random walk}
on an electrical network. This is a continuous time process $\{ X_t
\}_{t \in \R^+}$ which can be sampled by having the same transition
probabilities as the discrete time walk but introducing unit
exponential waiting times between transitions. (Contrast this with the
discrete time random walk, which we can think of as having waiting
times that are deterministically equal to $1$.)

In what follows, unless otherwise specified, all the electrical
networks we consider will have a distinguished vertex $v_0 \in V$, and
all random walks will be assumed to start at $v_0$.

\subsection{Local times} \label{subsec:local-times}

Let $X = \{ X_t \}_{t \in \Z^+}$ be a discrete time random walk on an
electrical network $G$. For each time $t$ and vertex $v$, we define
the quantity
\[ L^X_t(v) = \sum_{i = 0}^t \1_{\{ X_i = v \}}, \]
which counts the number of visits of $X$ to $v$ up to time $t$.

We also define a continuous analogue of $L^X_t(v)$. Suppose now that
$X = \{ X_t \}_{t \in \R^+}$ is a continuous time random walk on
$G$. For any time $t \ge 0$ and vertex $v \in V$, we define the
\emph{local time} $\L^X_t(v)$ of $X$ at $v$ to be
\[ \L^X_t(v) = \frac{1}{c_v} \int_0^t \1_{\{ X_s = v \}} ds. \]
Note the factor of $\frac{1}{c_v}$; this is a convenient normalization
for various formulas. When there is no risk of confusion about the
random walk $X$, we will sometimes shorten the notation to $L_t(v)$ or
$\L_t(v)$.

Clearly, the cover time is related to the local time; it is the first
time that all local times are positive. For a continuous time random
walk $X$, we have
\[ \taucov = \inf \left\{ t \ge 0 : \min_{x \in V} \L^X_t(x) > 0 \right\}. \]
We will also frequently consider the first time that $v_0$ accumulates
a certain amount of local time. We give a formal definition for this
stopping time. For a continuous time random walk $X$ and any $t > 0$,
define the \emph{inverse local time} $\tau^+(t)$ as
\[ \tau^+(t) = \inf \{ s \ge 0 : \L^X_s(v_0) \ge t \}, \]
It will always be clear what $X$ is, so it is not included in the
notation for sake of brevity.

\subsection{Gaussian free fields} \label{subsec:gff}

For an electrical network $G = (V, E)$, the Gaussian free field
$\eta_S$ with boundary $S \subset V$ is defined to be a random
variable taking values in the set $\R^{V \setminus S}$ of real-valued
functions on $V \setminus S$. Its probability density at an element $f
\in \R^{V \setminus S}$ is proportional to
\begin{equation} \label{eq:gff-def}
\exp \( - \frac{1}{4} \sum_{x, y \in V} c_{xy}(f(x) - f(y))^2 \),
\end{equation}
\noindent where we define $f(x) = 0$ for each $x \in S$. For our
purposes, Gaussian free fields will always have boundary $S = \{ v_0
\}$. Thus, if we refer to \emph{the} Gaussian free field on some
network, we will mean the one with this boundary, and we will drop the
subscript $S$.

From (\ref{eq:gff-def}) it is clear that $\eta$ is a multidimensional
Gaussian random variable. It is not too hard to calculate (e.g.,
Theorem 9.20 of \cite{J97}) that for all $x, y \in V$,
\[ \E\( \eta_x - \eta_y \)^2 = \Reff(x, y), \]
which confirms that our definition of the GFF is consistent with the
one given in the introduction. Noting that $\eta_{v_0} = 0$,
the above formula completely determines the correlations of $\eta$ in
terms of the effective resistances.

The Gaussian free field comes into the picture via a class of
identities known as Isomorphism Theorems. The first such theorems were
proved independently by Ray \cite{R63} and Knight \cite{K63} relating
the local times of Brownian motion to a $2$-dimensional Bessel
process. More generally, it turns out that for any strongly symmetric
Borel right process, there is an identity relating its local times to
an associated Gaussian process.

Inspired by formulas of Symanzik \cite{S66} and Brydges, Fr\"ohlich,
and Spencer \cite{BFS82}, Dynkin \cite{D84} gave the first isomorphism
of this type to be expressed in terms of Gaussian free fields. Various
related identities were subsequently discovered by Marcus and Rosen
\cite{MR92}, Eisenbaum \cite{E95}, Le Jan \cite{lJ10}, Sznitman
\cite{S12}, and others. There is a nice version of the isomorphism in
the case of continuous time random walks on finite electrical
networks, first appearing in \cite{EKMRS00} (see also Theorem 8.2.2 of
the book by Marcus and Rosen \cite{MR06}).

\begin{theorem}[Generalized Second Ray-Knight Theorem] \label{thm:second-ray-knight}
  Let $G = (V, E)$ be an electrical network, with a given vertex $v_0
  \in V$. Let $X = \{X_t\}_{t \ge 0}$ be a continuous time random walk
  on $G$, and for any $t > 0$, define $\tau^+(t) = \inf \{ s \ge 0 :
  \L^X_s(v_0) \ge t \}$ to be the first time that $v_0$ accumulates
  local time $t$. Then, we have

  \[ \left\{ \L^X_{\tau^+(t)}(x) + \frac{1}{2} \eta_x^2 \right\}_{x \in V} \laweq \left\{ \frac{1}{2}\(\eta_x + \sqrt{2t}\)^2 \right\}_{x \in V}. \]
\end{theorem}

\noindent For more background on isomorphism theorems, we refer the
interested reader to \cite{MR06}. See also \cite{lJ11} for information
relating Gaussian free fields to loop measures.

\subsection{Random walks on paths and the first Ray-Knight theorem} \label{subsec:path-walks}

The content of this subsection may appear somewhat unmotivated before
reading Section \ref{sec:domination}. The reader may wish to first
skim this subsection and revisit it when reading Section
\ref{subsec:discrete-local-times} where it is used.

We will need a few basic facts concerning the special case where the
underlying graph of $G$ is a path. In this setting, it is a classical
theorem proved independently by Ray and Knight that the local times of
a continuous time random walk can be related to Brownian motion.

\begin{theorem}[First Ray-Knight Theorem] \label{thm:first-ray-knight}
  For any $a > 0$, let $B_t$ be a standard one-dimensional Brownian
  motion started at $B_0 = a$, and let $T = \inf \{ t : B_t = 0
  \}$. Let $\{ W_t \}_{t \ge 0}$ be a standard two-dimensional
  Brownian motion. Then,
  \[ \Big\{ \L^{B_t}_T(x) : x \in [0, a] \Big\} \laweq \Big\{ |W_x|^2 : x \in [0, a] \Big\}, \]
  where $\L^{B_t}_T$ denotes the local time of Brownian motion.
\end{theorem}

In Section \ref{subsec:local-times}, we did not define the local time
of Brownian motion, which requires some minor technicalities due to
the fact that it can only be defined as a density. For background on
Brownian local times and Theorem \ref{thm:first-ray-knight}, we refer
the reader to \cite{MP10}. However, we will only use a discretized
version of Theorem \ref{thm:first-ray-knight}, where we restrict our
attention to a finite set of values for $x$. This is equivalent to
replacing the Brownian motion $B_t$ with a continuous time random walk
on a path.

\begin{corollary} \label{cor:first-ray-knight}
  Let $G = (V, E)$ be an electrical network whose underlying graph is
  a path, with vertices labeled $0, 1, 2, \ldots , N$ and conductances
  $c_{k, k + 1}$ between $k$ and $k + 1$ for $0 \le k < N$. Let $X_t$
  be a continuous time random walk on $G$ started at $X_0 = N$, and
  let $T = \inf \{ t : X_t = 0 \}$. Define
  \[ a_k = \sum_{i = 0}^{k - 1} \frac{1}{c_{i, i + 1}}, \]
  and let $\{ W_t \}_{t \ge 0}$ be a standard two-dimensional Brownian
  motion. Then,

  \[ \Big\{ \L^X_T(k) : 1 \le k < N \Big\} \laweq \Big\{ |W_{a_k}|^2 : 1 \le k < N \Big\}. \]
\end{corollary}
\begin{proof}
  The equivalence to Theorem \ref{thm:first-ray-knight} can be seen as
  follows. For any $x \in \R$, let $B_t$ be a Brownian motion started
  at $x$ and stopped upon hitting $x - r$ or $x + s$. Then, the local
  time accumulated at $x$ is distributed as an exponential random
  variable with mean $\frac{rs}{r + s}$.

  When $x = a_k$, $r = \frac{1}{c_{k, k - 1}}$ and $s = \frac{1}{c_{k,
      k + 1}}$, this corresponds to an exponential jump time from the
  vertex $k$ in $G$, scaled by a factor of $\frac{1}{c_{k, k - 1} +
    c_{k, k + 1}}$ which appears in the definition of $\L^X_T(k)$.
\end{proof}

In light of Corollary \ref{cor:first-ray-knight}, it is useful to know
something about two-dimensional Brownian motion. For our purposes, we
need the following estimate, which is a quantitative verson of the
standard fact that two-dimensional Brownian motion is not
point-recurrent.

\begin{lemma} \label{lem:brownian-disk-avoidance}
  Let $W_t$ be a standard two-dimensional Brownian motion. For any
  $\epsilon \in (0, 1)$ and $\lambda > 0$, we have

  \[ \P \( \inf_{\epsilon \le t \le 1} |W_t|^2 < \lambda \) \le \frac{2}{\log \epsilon^{-1}} + \frac{3}{\epsilon} \exp \( - \frac{\log \lambda^{-1}}{\log \epsilon^{-1}} \). \]
\end{lemma}
\begin{proof}
  See Appendix.
\end{proof}

Finally, the next lemma shows that certain conditioned random walks on
paths are equivalent to random walks on a path of different
conductances. Thus, the first Ray-Knight theorem may be applied in a
conditional setting as well. This will be important when we study
random walk transitions on general electrical networks.

\begin{lemma} \label{lem:conditioned-path}
  Let $N$ be a positive integer and $r > 0$ a real number.

  Consider an electrical network $G = (V, E)$ whose underlying graph
  is a path, with vertices labeled $0, 1, 2, \ldots , N + 1$. Suppose
  that the conductances are $c_{k, k + 1} = 1$ for $0 \le k < N$ and
  $c_{N, N + 1} = r$. Let $X = \{X_t\}_{t \ge 0}$ be a discrete time
  random walk on $G$ started at $N$, and let $\tau$ be the first time
  that $X$ hits $0$ or $N + 1$.

  On the other hand, let $G'$ be a path on vertices $0, 1, 2, \ldots ,
  N$ with conductances
  \[ c'_{k, k + 1} = \frac{\(N - k - 1 + \frac{1}{r}\)\(N - k + \frac{1}{r}\)}{\frac{1}{r}\(1 + \frac{1}{r}\)} \]
  for $0 \le k < N$. Let $Y = \{Y_t\}_{t \ge 0}$ be a discrete time
  random walk on $G'$ started at $k$. Then, the paths of $Y$ stopped
  upon hitting $0$ have the same distribution as the paths of $X$
  conditioned on $X_{\tau} = 0$.
\end{lemma}
\begin{proof}
  This can be easily checked by calculating hitting probabilities,
  which can then be used to calculate transition probabilities for
  $X_t$ conditioned on $X_\tau = 0$. See Appendix.
\end{proof}

\begin{corollary} \label{cor:conditioned-path}
  Let $N$, $r$, and $G$ be as in Lemma \ref{lem:conditioned-path}, and
  suppose further that $r < 1$. Let $X$ be a continuous time random
  walk on $G$, and let $\tau = \inf \{ t \ge 0 : X_t = 0 \text{ or } N
  + 1 \}$. Then, for any $\epsilon \in (0, 1)$ and $\beta > 0$,
  \[ \P \( \min_{\epsilon N \le k < N} \L^X_\tau(k) \le \beta N \mmid X_\tau = 0 \) \le \frac{2}{\log \epsilon^{-1} - C_\alpha } + \frac{C_\alpha}{\epsilon} \exp \( - \frac{\log \beta^{-1} - C_\alpha}{\log \epsilon^{-1} + C_\alpha} \) \]
  where $\alpha = rN$, and $C_\alpha > 0$ is a number depending only on $\alpha$.
\end{corollary}
\begin{remark}
The statement of Corollary \ref{cor:conditioned-path} takes this
somewhat awkward form because it will be used for $r$ on the order of
$\frac{1}{N}$.
\end{remark}
\begin{proof}
  By Lemma \ref{lem:conditioned-path} (using the same notation), the
  paths of $X$ are distributed as a random walk on a path of $N$
  edges with conductances
  \[ c'_{k, k + 1} = \frac{\(N - k - 1 + \frac{1}{r}\)\(N - k + \frac{1}{r}\)}{\frac{1}{r}\(1 + \frac{1}{r}\)} \]
  for $0 \le k < N$. Thus, by Corollary \ref{cor:first-ray-knight},
  \[ \P \( \min_{\epsilon N \le k < N} \L^X_\tau(k) \le \beta N \mmid X_\tau = 0 \) = \P \( \min_{\epsilon N \le k < N} |W_{a_k}|^2 \le \beta N \), \]
  where $W_t$ is a two-dimensional Brownian motion, and
  \[ a_k = \sum_{i = 0}^{k - 1} \frac{1}{c'_{i, i + 1}} = \sum_{i = 0}^{k - 1} \frac{1}{r} \( 1 + \frac{1}{r} \) \( \frac{1}{N - i - 1 + \frac{1}{r}} - \frac{1}{N - i + \frac{1}{r}} \) \]
  \[ = \frac{1}{r} \( 1 + \frac{1}{r} \) \( \frac{1}{N - k + \frac{1}{r}} - \frac{1}{N + \frac{1}{r}} \). \]
  From the above equations, the following bounds are easy to verify
  for $\epsilon N \le k < N$.
  \[ c'_{k - 1, k + 1} + c'_{k, k + 1} \ge 2 \]
  \[ a_k \ge \frac{1}{r} \( 1 + \frac{1}{r} \) \( \frac{1}{N - \epsilon N + \frac{1}{r}} - \frac{1}{N + \frac{1}{r}} \) > \frac{\epsilon N}{(1 + rN)^2}. \]
  \[ a_k \le a_N \le \frac{2}{r}. \]
  It follows that
  \[ \P \( \min_{\epsilon N \le k < N} \L^X_\tau(k) \le \beta N \mmid X_\tau = 0 \) \le \P \( \inf_{\frac{\epsilon N}{(1 + rN)^2} \le t \le \frac{2}{r}} |W_t|^2 \le \beta N \) \]
  \[ = \P \( \inf_{\frac{\epsilon rN}{2 (1 + rN)^2} \le t \le 1} |W_t|^2 \le \frac{\beta r N}{2} \) = \P \( \inf_{\frac{\epsilon \alpha}{2 (1 + \alpha)^2} \le t \le 1} |W_t|^2 \le \frac{\beta \alpha}{2} \) \]
  \[ \le \frac{2}{\log \epsilon^{-1} - C_\alpha } + \frac{C_\alpha}{\epsilon} \exp \( - \frac{\log \beta^{-1} - C_\alpha}{\log \epsilon^{-1} + C_\alpha} \), \]
  for $C_\alpha$ sufficiently large. In the second line, we have used
  the scale-invariance of Brownian motion, and the third line is an
  application of Lemma \ref{lem:brownian-disk-avoidance}.
\end{proof}

\section{Stochastic domination in the generalized second Ray-Knight theorem} \label{sec:domination}

The goal of this section is to prove the following stochastic
domination theorem, which is a variant of Theorem 3 in \cite{L14}.

\begin{theorem}[variant of \cite{L14}, Theorem 3] \label{thm:domination}
  Let $\tau^+(t)$ and $\eta$ be as in Theorem
  \ref{thm:second-ray-knight}. Then, we have
  \[ \left\{ \sqrt{\L_{\tau^+(t)}(x)} : x \in V \right\} \preceq \frac{1}{\sqrt{2}} \left\{ \max \( \eta_x + \sqrt{2t}, 0 \) : x \in V \right\}, \]
  where $\preceq$ denotes stochastic domination.
\end{theorem}

Theorem \ref{thm:domination} extends Theorem 2.3 from \cite{D14},
which proves the result for trees. The approach in \cite{D14} uses a
Markovian property of local times for trees which does not seem to
extend to general electrical networks. We take a different approach of
embedding the finite-dimensional Gaussian free field inside a larger
infinite-dimensional Gaussian free field, which has desirable
continuity properties that were not apparent in the finite-dimensional
setting. As mentioned in the introduction, we discovered while writing
up our results that this idea appeared earlier in \cite{L14}.

Let us first give a heuristic description of the approach. Recall that
the continuous time random walk on an electrical network makes jumps
at exponentially distributed random intervals. An equivalent way of
sampling the continuous time random walk is to perform a Brownian
motion along the edges of the network. By this we mean that our
discrete state space $V$ is replaced by a larger state space
$\widehat{V}$ which includes not only the vertices in $V$ but also
each point along each edge of $E$ (regarding the edges as line
segments, so that $\widehat{V}$ is topologically a simplicial
$1$-complex). The object $\widehat{V}$ is known as a \emph{metric
  graph} and arises in physics and chemistry (see e.g. \S 5 of
\cite{CDT05}).

A Brownian motion on $\widehat{V}$ is, informally, a continuous Markov
process $\widehat{X} = \{ \widehat{X}(t) \}_{t \ge 0}$ taking values
in $\widehat{V}$ that behaves like a one-dimensional Brownian motion
on edges. The earliest rigorous development of this idea we could find
was carried out by Baxter and Chacon \cite{BC84}. See also
\cite{KPS12} for a more recent treatment.

It turns out that the Gaussian free field $\widehat{\eta}$ on
$\widehat{V}$ (without defining this precisely) is almost surely
continuous in the topology of $\widehat{V}$.\footnote{The Gaussian
  free field on $\widehat{V}$ can be constructed by sampling the GFF
  on $V$ and then sampling Brownian bridges on each edge.} We can also
define a notion of local time $\L^{\widehat{X}}_t(v)$, and we can
define the stopping time $\tau^+(t)$ analogously to the discrete
case. For convenience, let us write $\widehat{\L}_t$ for
$\L^{\widehat{X}}_{\tau^+(t)}$. With an appropriate normalization, the
restrictions of $\widehat{\eta}$ and $\widehat{\L}_t$ to $V \subset
\widehat{V}$ have the same laws as the corresponding objects on the
original network $G = (V, E)$. The generalized second Ray-Knight
theorem translates to
\begin{equation} \label{eq:continuous-isomorphism}
  \left\{ \widehat{\L}_{\tau^+(t)}(v) + \frac{1}{2}
  \widehat{\eta}_v^2 : v \in \widehat{V} \right\} \laweq \left\{
  \frac{1}{2} \( \widehat{\eta}'_v + \sqrt{2t} \)^2 : v \in
  \widehat{V} \right\},
\end{equation}
where $\widehat{\eta}'$ is another copy of $\widehat{\eta}$, and $c_v$
is a continuous analogue of the total conductance at a vertex.

Now, suppose that $\widehat{\eta}$ and $\widehat{\eta}'$ are coupled
in a way so that the two sides in equation
\ref{eq:continuous-isomorphism} are actually equal. Consider the
function $f: \widehat{V} \to \R$ given by $f(x) = (\widehat{\eta}'_x +
\sqrt{2t}) - \widehat{\eta}_x$. We have that $f(v_0) =
\sqrt{2t} > 0$, $f$ is continuous, and if $f(x) = 0$,
then $\widehat{\L}_t(x) = 0$. It turns out that the set $U = \{ v \in
\widehat{V} : \widehat{\L}_t(v) > 0 \}$ is connected, and clearly it
includes $v_0$. It follows that $f(x) > 0$ for all $x \in U$, which is
exactly the desired stochastic domination once we restrict to $V
\subset \widehat{V}$.

The assertion that $U$ is connected deserves some elaboration. It is
intuitively clear that the closure of $U$ should be connected, since
any point $v \in \widehat{V}$ which accumulates positive local time
must have been visited along some connected path from $v_0$ to
$v$. Thus, every non-trivial segment along this path should have also
accumulated positive local time.

On the other hand, it is not immediately obvious why $U$ itself is
connected, since there might be local times of $0$ at isolated
points. However, we can see heuristically that this pathology doesn't
occur by the first Ray-Knight theorem. Recall from Section
\ref{subsec:path-walks} that the first Ray-Knight theorem equates the
local times of a certain stopped Brownian motion to the distance of a
planar Brownian motion from the origin. Because planar Brownian motion
is not point-recurrent, the local times are \emph{all} positive almost
surely, and in particular, the set of points with $0$ local time does
not have isolated points.

To avoid technicalities, we will not actually use Brownian motion in
our proof. Instead, we will use a discrete approximation of Brownian
motion and pass to the limit. Arguments involving the continuity of
Gaussian free fields and positivity of local times will be translated
into corresponding quantitative estimates.

\subsection{A discrete refinement of $G$} \label{subsec:discrete-refinement}

Recall our setting of an electrical network $G = (V, E)$ with
conductances $\{ c_{xy} : x, y \in V \}$. For each positive integer $N
> 1$, we define a refinement $G_N = (V_N, E_N)$ by replacing each edge
$(x, y) \in E$ with a length $N$ path whose vertices we denote by

\[ \{ x = v_{xy, 0}, v_{xy, 1}, \ldots , v_{xy, N} = y \}. \]

We thus have edges between $v_{xy, i}$ and $v_{xy, i + 1}$ for each $0
\le i < N$. We will use $v_{yx, i}$ to denote the same vertex as
$v_{xy, N - i}$, and we will regard $V$ as a subset of $V_N$, so that
a vertex $x \in V$ will sometimes be considered as a vertex in $V_N$.

We choose the conductances of $G_N$ so that the effective resistance
between $x, y \in V$ as vertices in $G$ will be the same when they are
considered as vertices in $G_N$. In particular, we set the conductance
between $v_{xy, i}$ and $v_{xy, i + 1}$ to be $N c_{xy}$. Since the
effective resistances are equivalent, $G$ is in some sense a
projection of $G_N$. The following proposition makes this explicit.

\begin{proposition} \label{prop:projection}
Let $\eta$ be the GFF on $G$, and let $X$ be a continuous time random
walk on $G$. Let $\eta_N$ and $X_N$ denote the corresponding objects
for $G_N$. Then, for any $t > 0$ we have the following two identities
in law.

\[ \{ \eta_{N, v} : v \in V \} \laweq \{ \eta_v : v \in V \} \]
\[ \left\{ \L^{X_N}_{\tau^+(t)}(x) : x \in V \right\} \laweq \Big\{ \L^X_{\tau^+(t)}(x) : x \in V \Big\}. \]
\end{proposition}

The identity between $\eta_N$ and $\eta$ is immediate from the
equivalence of effective resistances. The identity between local times
then follows from Theorem \ref{thm:second-ray-knight}. However, there
is also a very direct way to see the equivalence of local times which
we now describe.

If $X_N(t)$ is a continuous time random walk on $G_N$ started at
$v_0$, then $X_N(t)$ induces a random walk $X_N^G(t)$ on $G$ by only
recording the time spent in $V$. More formally, define $t_0 = 0$, and
for each $i \ge 0$, define
\[ t_{i + 1} = \inf \{ t > t_i : X_N(t) \in V \text{ and } X_N(t)
\ne X_N(t_i)\}.\footnote{We are taking our process $X_N$ to be right
  continuous, so the infimum is achieved, and in particular $X_N(t_{i
    + 1}) \in V$.} \]
Define also
\[ s_i = \int_{t_i}^{t_{i + 1}} \1_{\{ X_N(s) = X_N(t_i) \}} ds \]
to be the amount of time spent in $X_N(t_i)$ during the time interval
$[t_i, t_{i + 1}]$.

Then, consider the $V$-valued process $X_N^G(t)$ which starts at $v_0$
and, for each $i$, jumps to $X_N(t_{i + 1})$ at time $\sum_{j = 1}^i
s_j$. Note that if $X_N(t_i) = x \in V$, at the next jump $X_N$
transitions to $v_{xy, 1}$ with probability $\frac{c_{xy}}{c_x}$ for
each $y$ neighboring $x$ in $G$. After that, $X_N$ behaves like a
simple random walk on $\Z$ started at $1$ and stopped upon hitting
either $0$ (corresponding to $v_{xy, 0} = x$) or $N$ (corresponding to
$v_{xy, N} = y$). Thus, with probability $\frac{N - 1}{N}$ it returns
to $x$, and with probability $\frac{1}{N}$ it hits $y$.

Consequently, between times $t_i$ and $t_{i + 1}$, the number of times
$X_N$ visits $x$ is geometrically distributed with mean $N$, and so
the accumulated local time $s_i$ is exponentially distributed with
mean $N$. Moreover, we see that
\[ \P \( X_N(t_{i + 1}) = y \mmid X_N(t_i) = x \) = \frac{c_{xy}}{c_x}, \]
so $X_N^G(t)$ has the same law as a continuous time random walk on $G$
except that the waiting times between jumps are scaled by $N$. In
particular, we have
\[ \left\{ \L^{X_N}_{\tau^+(t)}(x) : x \in V \right\} = \left\{ \frac{1}{N} \cdot \L^{X^G_N}_{\tau^+(Nt)}(x) : x \in V \right\} \laweq \Big\{ \L^X_{\tau^+(t)}(x) : x \in V \Big\}, \]
where $X$ is a continuous time random walk on $G$. Note that the
factor of $N$ appearing in the middle expression comes from the
normalization by total conductance at $x$, which differs for $G$ and
$G_N$.

\subsection{Local times of $G_N$} \label{subsec:discrete-local-times}

We will need two estimates concerning local times on $G_N$, stated as
Lemmas \ref{lem:near-local-time} and \ref{lem:bridge-local-time}
below. These correspond to our assertion that the set $U$ is connected
in the heuristic proof outline provided at the beginning of the
section.

In the lemmas that follow, we consider a continuous time random walk
$X_N(t)$ on $G_N$ started at a vertex $x \in V$. Let $\tau_x$ denote
the first time the walk hits another vertex $y \in V$ distinct from
$x$. The first estimate states, roughly, that it is very likely for
vertices near $x$ to accumulate significant local time.

We will need a standard concentration estimate for sums of
i.i.d. exponential random variables. Unfortunately, we were unable to
find a reference that contained both tail bounds, so a short proof is
included in the appendix.

\begin{lemma} \label{lem:exp-conc}
  Let $X_1, X_2, \ldots , X_N$ be i.i.d. exponential random variables
  with mean $\mu$. Then, for any $\alpha \in [0, 1]$, we have
  \[ \P\( \left| \sum_{i = 1}^N X_i - \mu N \right| \ge \alpha \mu N \) \le 2 e^{-\frac{1}{4} \alpha^2 N}. \]
\end{lemma}
\begin{proof}
  See Appendix.
\end{proof}

\begin{lemma} \label{lem:near-local-time}
  Let $y \in V$ be any neighbor of $x$ in $G$, let $\epsilon \in \(0,
  \frac{1}{2}\), \lambda > 0$ be given, and define $k = \lfloor
  \epsilon N \rfloor$. Then,
  \[ \P\( \min_{0 \le i \le k} \L_{\tau_x}(v_{xy, i}) < \lambda \) \le C_G \cdot \epsilon N \( \lambda + \exp \( - \frac{\lambda N}{8 C_G} \) \) \]
  for some constant $C_G$ depending on $G$ but not $N$.
\end{lemma}
\begin{proof}
  Recall the notation $L_{\tau_x}(x)$ for the number of visits to $x$
  up until time $\tau_x$, and recall also from Section
  \ref{subsec:discrete-refinement} that $L_{\tau_x}(x)$ is distributed
  as a geometric random variable with mean $N$. Conditioning on
  $L_{\tau_x}(x)$, we may decompose the walk up until time $\tau_x$
  into $L_{\tau_x}(x)$ excursions from $x$ and a path to a neighbor of
  $x$ in $G$. Each excursion may be sampled independently.

  Let us now consider one excursion. The first step of the excursion
  goes to some vertex $v_{xz, 1}$, where $z$ is a neighbor of $x$ in
  $G$. As noted earlier, from there the walk behaves like a simple
  random walk on $\Z$ started at $1$, stopped upon hitting $0$
  (corresponding to the return to $x$), and conditioned on hitting $0$
  before $N$ (corresponding to avoiding $z$).

  Let $E_m$ denote the event that a simple random walk on $\Z$ started
  at $1$ hits $m$ before $0$. By a standard martingale argument, we
  have $\P(E_m) = \frac{1}{m}$. Thus,

  \[ \P \( E_k \mmid E_N^c \) = \frac{\P(E_k \cap E_N^c)}{\P(E_N^c)} \ge \frac{1}{k} - \frac{1}{N} > \frac{1}{2k}. \]

  In particular, this implies that for each excursion, there is a
  $\frac{c_{xy}}{c_x}$ probability that the first step is $v_{xy, 1}$,
  and with probability at least $\frac{1}{2k}$ the excursion will then
  hit $v_{xy, k}$. In other words, letting $p$ be the probability that
  a single excursion includes $v_{xy, k}$, we have $p \ge
  \frac{c_{xy}}{2kc_x}$.

  Let $L$ denote the number of excursions which hit $v_{xy, k}$. By
  the preceding discussion, it is the sum of $L_{\tau_x}(x)$ i.i.d. Bernoulli
  random variables with expectation $p$. Since $L_{\tau_x}(x)$ is geometrically
  distributed with mean $N$, it follows that $L$ is geometrically
  distributed with mean $pN$. We thus have
  \begin{equation} \label{eq:P(L small)}
    \P \( L < 2 \lambda c_x N \) \le \frac{2 \lambda c_x N}{pN} \le \frac{4 \lambda c_x^2 \epsilon N}{c_{xy}}.
  \end{equation}

  Note that for each $i \in \{ 0, 1, 2, \ldots, k \}$, the vertex
  $v_{xy, i}$ is visited at least $L$ times, and the total conductance
  of $v_{xy, i}$ is at most $Nc_x$. Thus, $\L_{\tau_x}(v_{xy}, i)$
  stochastically dominates $\frac{1}{Nc_x}$ times the sum of $L$
  i.i.d. unit exponentials. By Lemma \ref{lem:exp-conc} with $\alpha =
  \frac{1}{2}$, we have
  \[ \P \Big( c_x N \cdot \L_{\tau_x}(v_{xy, i}) < \lambda c_x N \,\Big|\, L \ge 2 \lambda c_x N \Big) \le 2 \exp \( - \frac{\lambda c_x N}{8} \), \]
  and so
  \[ \P \( \min_{0 \le i \le k} \L_{\tau_x}(v_{xy, i}) < \lambda \mmid L \ge 2 \lambda c_x N \) \le 2 \epsilon N \exp \( - \frac{\lambda c_x N}{8} \). \]
  Combining this with equation (\ref{eq:P(L small)}) gives
  \[ \P \( \min_{0 \le i \le k} \L_{\tau_x}(v_{xy, i}) < \lambda \) \le \frac{4 \lambda c_x^2 \epsilon N}{c_{xy}} + 2 \epsilon N \exp \( - \frac{\lambda c_x N}{8} \), \]
  which takes the desired form for $C_G$ sufficiently large.
\end{proof}

\begin{corollary} \label{cor:near-local-time}
  Let $S = \{ y \in V : (x, y) \in E \}$ be the set of neighbors of
  $x$ in $G$. Then,
  \[ \P\( \min_{y \in S} \min_{0 \le k \le \frac{N}{\log^3 N}} \L_{\tau_x}(v_{xy, k}) < \frac{\log^2 N}{N} \) \longrightarrow 0 \]
  as $N \rightarrow \infty$.
\end{corollary}
\begin{proof}
  This follows immediately from Lemma \ref{lem:near-local-time} by
  taking $\lambda = \frac{\log^2 N}{N}$ and $\epsilon =
  \frac{1}{\log^3 N}$.
\end{proof}

The second estimate states that, conditioned upon $X_N(\tau_x) = y$,
it is very likely that vertices $v_{xy, k}$ are visited a large number
of times, as long as $k$ is not too close to $N$. This essentially
follows from Corollary \ref{cor:conditioned-path} from Section
\ref{subsec:path-walks}.

\begin{lemma} \label{lem:bridge-local-time}
  Let $y$ be a neighbor of $x$ in $G$. Then, for any $\epsilon,
  \lambda \in (0, 1)$, we have
  \[ \P\( \min_{\epsilon N \le k < N} \L_{\tau_x}(v_{yx, k}) < \lambda \mmid X_N(\tau_x) = y \) \le \frac{2}{\log \epsilon^{-1} - C_G} + \frac{C_G}{\epsilon} \exp \( - \frac{\log \lambda^{-1} - C_G }{\log \epsilon^{-1} + C_G} \) \]
  for some constant $C_G$ depending on $G$ but not $N$.
\end{lemma}
\begin{proof}
  Let $S = \{ z \in V : (x, z) \in E \}$. Note that the process $X_N$
  up to time $\tau_x$ induces a continuous time random walk $Y = \{
  Y_t \}_{t \ge 0}$ on the vertices
  \[ \{ v_{xy, 0}, v_{xy, 1}, \ldots , v_{xy, N} \} \cup S \]
  by ignoring visits to vertices outside of that set (namely, those of
  the form $v_{xz, k}$ for $z \ne y$ and $1 \le k < N$). We can define
  a stopping time $T_x$ analogous to $\tau_x$ as the first time $Y$
  hits $S$.

  For convenience, define $p_{xz} = \frac{c_{xz}}{c_x}$ for each $z
  \in S$. Note that
  \[ \P\( \text{$X_N$ hits $v_{xy, 1}$ before hitting $S$ or returning to $x$} \) = p_{xy} \]
  \[ \P\( \text{$X_N$ hits $S$ before hitting $v_{xy, 1}$ or returning to $x$} \) = \frac{1 - p_{xy}}{N}. \]
  Thus, we can interpret $Y$ up to time $T_x$ as a continuous time
  random walk on a path with vertices $(w_0, w_1, w_2, \ldots , w_{N +
    1})$, where all the conductances are $1$ except that the
  conductance between $w_N$ and $w_{N + 1}$ is $\frac{1 -
    p_{xy}}{Np_{xy}}$. Here, $w_k$ corresponds to $v_{yx, k}$ (so $Y$
  is started at $w_N$), and $w_{N + 1}$ corresponds to any vertex in
  $S \setminus \{ y \}$ (we may combine all of these states because
  $Y$ is stopped upon hitting this set anyway).

  We are now in the setting of Corollary \ref{cor:conditioned-path},
  as conditioning on $Y_{T_x} = y$ corresponds to conditioning on
  hitting $w_0$ before $w_{N + 1}$. Following the notation of
  Corollary \ref{cor:conditioned-path}, we have $r = \frac{1 -
    p_{xy}}{Np_{xy}}$, so that $\alpha = \frac{1 - p_{xy}}{p_{xy}}$.

  We apply the corollary with $\beta = \lambda c_{xy}$. Note
  that the total conductances at $v_{yx, k}$ are $2Nc_{xy}$ as opposed
  to $2$ in the statement of Corollary \ref{cor:conditioned-path}, so
  the local times will be scaled accordingly. It follows that
  \[ \P\( \min_{\epsilon N \le k < N} \L^{X_N}_{\tau_x}(v_{yx, k}) < \lambda \mmid X_N(\tau_x) = y \) = \P\( \min_{\epsilon N \le k < N} \L^Y_{T_x}(w_k) < \lambda c_{xy} N \mmid Y_{T_x} = y \) \]
  \[ \le \frac{2}{\log \epsilon^{-1} - C_\alpha} + \frac{C_\alpha}{\epsilon} \exp \( - \frac{\log \lambda^{-1} - \log c_{xy} - C_\alpha }{\log \epsilon^{-1} + C_\alpha} \) \]
  \[ \le \frac{2}{\log \epsilon^{-1} - C_G} + \frac{C_G}{\epsilon} \exp \( - \frac{\log \lambda^{-1} - C_G }{\log \epsilon^{-1} + C_G} \), \]
  whenever $C_G > \max(C_\alpha, C_\alpha + \log c_{xy})$. In
  particular, since there are only finitely many possible values of
  $p_{xy}$ and hence of $\alpha$, we can choose $C_G$ sufficiently
  large so that this holds independently of $N$. This proves the
  lemma.
\end{proof}

\begin{corollary} \label{cor:bridge-local-time}
  Let $y$ be a neighbor of $x$ in $G$. Then, we have
  \[ \P\( \min_{\frac{N}{\log^3 N} \le k \le N} \L_{\tau_x}(v_{yx, k}) < \frac{\log^2 N}{N} \mmid X_{\tau_x} = y \) \longrightarrow 0 \]
  as $N \rightarrow \infty$.
\end{corollary}
\begin{proof}
  We apply Lemma \ref{lem:bridge-local-time} with $\epsilon =
  \frac{1}{\log^3 N}$ and $\lambda = \frac{\log^2 N}{N}$. It
  suffices to show that both terms on the right hand side tend to
  zero. Clearly,
  \[ \frac{2}{\log \epsilon^{-1} - C_G} \rightarrow 0 \]
  as $N \rightarrow \infty$. To bound the other term, note that for
  sufficiently large $N$, we have
  \[ \frac{\log \lambda^{-1} - C_G}{\log \epsilon^{-1} + C_G} = \frac{\log N - 2 \log \log N - C_G}{3 \log \log N + C_G} \ge \frac{\log N}{6 \log \log N}, \]
  in which case
  \[ \frac{C_G}{\epsilon} \exp \( - \frac{\log \lambda^{-1} - C_G}{\log \epsilon^{-1} + C_G} \) \le C_G \log^3 N \exp \( - \frac{\log N}{6 \log \log N} \) \]
  \[ = C_G \exp \( - \frac{\log N}{6 \log \log N} + 3 \log \log N \) \longrightarrow 0. \]
\end{proof}

\subsection{Proof of Theorem \ref{thm:domination}}

We now prove Theorem \ref{thm:domination}, following the plan outlined
at the beginning of the section. Let us first prove an approximation
of Theorem \ref{thm:domination}.

\begin{lemma} \label{lem:domination}
  Let $t > 0$ be given. Let $\Omega_N$ be a probability space with
  random variables $\eta_N$, $\eta'_N$, and $X_N = \{ X_N(t) \}_{t \ge
    0}$ such that $\eta_N$ and $\eta'_N$ are distributed as Gaussian
  free fields on $G_N$, and $X_N$ is distributed as a continuous time
  random walk on $G_N$. Furthermore, suppose that $\eta_N$ and $X_N$
  are independent, and almost surely for each $v \in V_N$,

  \[ \frac{1}{2} \eta_{N, v}^2 + \L^{X_N}_{\tau^+(t)}(v) = \frac{1}{2} \( \eta'_{N, v} + \sqrt{2t} \)^2. \]

  \noindent (Theorem \ref{thm:second-ray-knight} ensures that such a
  construction is always possible.) Then, for any $\epsilon > 0$, we
  have

  \[ \P \( \text{for some $x \in V$, both $\L^{X_N}_{\tau^+(t)}(x) > 0$ and $\eta'_{N, x} + \sqrt{2t} < 0$} \) \le \epsilon \]

  \noindent for $N$ sufficiently large.
\end{lemma}
\begin{remark} \label{rmk:domination}
  Note that the hypothesis of Lemma \ref{lem:domination} implies for
  each $x \in V$ that
  \[ \sqrt{\L^{X_N}_{\tau^+(t)}(x)} \le \frac{1}{\sqrt{2}} \left| \eta'_{N, x} + \sqrt{2t} \right|. \]
  Consequently, the conclusion of the lemma may be expressed
  equivalently as
  \[ \P \( \sqrt{\L^{X_N}_{\tau^+(t)}(x)} > \frac{1}{\sqrt{2}} \max \( 0, \eta'_{N, x} + \sqrt{\frac{2t}{c_{v_0}}} \) \text{ for some $x \in V$} \) \le \epsilon. \]
\end{remark}

\begin{proof}
  To shorten notation, we use $\tau^+$ to denote $\tau^+(t)$.

  Call a vertex $x \in V$ \emph{well-connected} at time $s$ if there
  exists a sequence of vertices $v_0 = w_0, w_1, \ldots , w_n = x$ in
  $V_N$ such that $(w_i, w_{i + 1}) \in E_N$ and $\L_s^{X_N}(w_i) \ge
  \frac{\log^2 N}{N}$ for each $i$. We will show that with high
  probability, every vertex in $V$ with positive local time at time
  $\tau^+$ is well-connected.

  Recall from the discussion in Section
  \ref{subsec:discrete-refinement} that $X_N$ induces a random walk on
  $G$ which, when regarded as a sequence of visited vertices
  (disregarding holding times), has the same law as a discrete time
  random walk on $G$. Thus, one way of sampling from $X_N$ is to first
  sample a path
  \[ P = (v_0 = x_0, x_1, x_2, \ldots ) \]
  of the discrete time random walk on $G$. Then, we construct $X_N$ as
  follows. For each $i \ge 0$, let $Y_i(t)$ be a continuous time
  random walk on $G_N$ started at $x_i$, and let $\tau_i$ be the first
  time that $Y_i$ hits a neighbor of $x_i$ in $G$.

  Let $Z_i$ have the law of a copy of $Y_i$ conditioned on the event
  $Y_i(\tau_i) = x_{i + 1}$. Then, we may form $X_N$ by concatenating
  the walks $Z_i$ up to time $\tau_i$. More formally, we may define

  \[ n(s) = \max \left\{ n \ge 1 : \sum_{i = 1}^{n - 1} \tau_i \le s \right\} \]
  and set $X_N(s) = Z_{n(s)}\( s - \sum_{i = 1}^{n(s) - 1} \tau_i \)$.

  To lighten notation, let us write $\L_i = \L^{Y_i}_{\tau_i}$ and
  $\P_i( \cdot ) = \P \( \cdot \mmid Y_i(\tau_i) = x_{i + 1}\)$,
  noting that the randomness of the $Y_i$ are independent. Let $P(s) =
  (x_1, x_2, \ldots , x_{n(s)})$ denote the truncation of $P$ up until
  time $s$. We will say that $P(s)$ is well-connected if each $x_i$
  appearing in $P(s)$ is well-connected at time $s$. Then,
  \begin{align}
    & \P \Big( \text{$P(\tau^+)$ is not well-connected} \,\Big|\, P(\tau^+) \Big) \nonumber \\
    & \longeqpad \le \sum_{i = 1}^{|P(\tau^+)| - 1} \P_i \( \min_{0 \le k \le N} \L_i(v_{x_ix_{i + 1}, k}) < \frac{\log^2 N}{N} \) \nonumber \\
    & \longeqpad = \sum_{i = 1}^{|P(\tau^+)| - 2} \P_i \( \min_{0 \le k \le N - \frac{N}{\log^3 N}} \L_i(v_{x_ix_{i + 1}, k}) < \frac{\log^2 N}{N} \) + \nonumber \\
    & \longeqpad\hphantom{\;=\;} \sum_{i = 2}^{|P(\tau^+)| - 1} \P_i \( \min_{0 \le k < \frac{N}{\log^3 N}} \L_i(v_{x_ix_{i - 1}, k}) < \frac{\log^2 N}{N} \) \label{eq:path-decomp}
  \end{align}

  Fix a number $T$ sufficiently large so that $\P \Big( |P(\tau^+)| >
  T \Big) \le \frac{\epsilon}{4}$. Again, by the discussion of Section
  \ref{subsec:discrete-refinement}, the law of $P(\tau^+)$ does not
  depend on $N$, so the number $T$ can be chosen independently of
  $N$. Note that by Corollaries \ref{cor:near-local-time} and
  \ref{cor:bridge-local-time}, each summand in either sum of the last
  expression of (\ref{eq:path-decomp}) is bounded by
  $\frac{\epsilon}{8T}$ for sufficiently large $N$. Consequently, for
  sufficiently large $N$, the whole expression is bounded by $2
  |P(\tau^+)| \cdot \frac{\epsilon}{8T}$, and we have
  \begin{align*}
    \P \Big( \text{$P(\tau^+)$ is not well-connected} \Big) \le &\; \P
    \Big( |P(\tau^+)| > T \Big) + \\
    &\; \P \Big( \text{$P(\tau^+)$ is not well-connected} \,\Big|\, |P(\tau^+)| \le T \Big) \\
    \le &\; \frac{\epsilon}{4} + 2T \cdot \frac{\epsilon}{8T} = \frac{\epsilon}{2}.
  \end{align*}

  \noindent Note that almost surely, the vertices $x \in V$ for which
  $\L_{\tau^+}(x) > 0$ are exactly those appearing in
  $P(\tau^+)$. Thus, we have

  \begin{equation} \label{eq:well-connected}
    \P \Big( \text{for some $x \in V$, $\L_{\tau^+}(x) > 0$ but $x$ is not well-connected} \Big) \le \frac{\epsilon}{2}.
  \end{equation}

  We next show that with high probability, the values of $\eta'_N$ at
  adjacent vertices do not differ by very much. Consider any $(x, y)
  \in E$ and $0 \le k < N$. For notational convenience, let $u =
  v_{xy, k}$ and $w = v_{xy, k + 1}$. We have

  \[ \E (\eta'_{N, u} - \eta'_{N, w})^2 = \Reff(u, w) \le \frac{1}{N c_{xy}}. \]

  \noindent Since $\eta'_{N, u} - \eta'_{N, w}$ has a Gaussian
  distribution, it follows that

  \[ \P \( |\eta'_{N, u} - \eta'_{N, w}| \ge \frac{\log N}{\sqrt{N}} \) \le \exp \( - c_{xy} \log^2 N \). \]

  \noindent Taking a union bound over all adjacent pairs $(u, w) \in
  E_N$, we obtain

  \begin{equation} \label{eq:gff-continuity}
    \P \( \max_{(u, w) \in E_N} |\eta'_{N, u} - \eta'_{N, w}| \ge
    \frac{\log N}{\sqrt{N}} \) \le N \exp \( - \( \min_{(x, y) \in E} c_{xy} \) \log^2 N \) \le
    \frac{\epsilon}{2}
  \end{equation}

  \noindent for $N$ sufficiently large.

  Finally, we may combine equations (\ref{eq:well-connected}) and
  (\ref{eq:gff-continuity}) to deduce the lemma. Indeed, suppose that
  for some $x \in V$, we have $\L^{X_N}_{\tau^+}(x) > 0$ but
  $\sqrt{2t} + \eta'_{N, x} < 0$. If $x$ is well-connected at time
  $\tau^+$, which occurs with high probability by
  (\ref{eq:well-connected}), then there exists a path $v_0 = w_0, w_1,
  \ldots , w_n = x$ in $G_N$ such that each $\L^{X_N}_{\tau^+}(w_i)$
  is at least $\frac{\log^2 N}{N}$. Observe that $\sqrt{2t} +
  \eta'_{N, v_0} = \sqrt{2t} > 0$, so for some $i$ we must have

  \[ \sqrt{2t} + \eta'_{N, w_i} > 0 \text{ and } \sqrt{2t} + \eta'_{N, w_{i + 1}} < 0. \]

  \noindent However, we also have

  \[ \frac{1}{\sqrt{2}} \left| \sqrt{2t} + \eta'_{N, w_i} \right| = \sqrt{\L^{X_N}_{\tau^+}(w_i) + \frac{1}{2} \eta_{N, x_i}^2} \ge \frac{\log N}{\sqrt{N}}. \]

  \noindent Therefore, this can only happen if

  \[ \left| \eta'_{N, w_i} - \eta'_{N, w_{i + 1}} \right| \ge \frac{2 \log N}{\sqrt{N}}. \]

  \noindent But by equation (\ref{eq:gff-continuity}), this is
  unlikely. Thus, we have
  \begin{align*}
  & \P \( \text{for some $v \in V$, both $\L^{X_N}_{\tau^+}(v) > 0$ and $\sqrt{2t} + \eta'_{N, v} < 0$} \) \\
  & \longeqpad \le \P \( \max_{(u, w) \in E_N} |\eta'_{N, u} - \eta'_{N, w}| \ge \frac{\log N}{\sqrt{N}} \) + \\
  & \longeqpad \hphantom{\;\le\;} \P \Big( \text{for some $x \in V$, $\L_{\tau^+}(x) > 0$ but $x$ is not well-connected} \Big) \\
  & \longeqpad \le \frac{\epsilon}{2} + \frac{\epsilon}{2} = \epsilon,
  & \end{align*}

  \noindent proving the lemma.
\end{proof}

\noindent Theorem \ref{thm:domination} is now an easy consequence of Lemma \ref{lem:domination}.

\begin{proof}[Proof of Theorem \ref{thm:domination}]
  Let $A \subset \R^V$ be any monotone set. Let $\epsilon > 0$ be
  given, and take $N$ sufficiently large so that the conclusion of
  Lemma \ref{lem:domination} holds.

  Let $\eta_N$ be the Gaussian free field on $G_N$, and let $X_N$ be a
  continuous time random walk independent of $\eta_N$. We will now try
  to define another Gaussian free field $\eta'_{N, v}$ on the same
  probability space so as to satisfy the hypotheses of Lemma
  \ref{lem:domination}. In fact, by the isomorphism theorem, $\eta'_N$
  can be given in terms of $\eta_N$ and the local times up to a choice
  of sign in taking the square root.

  To determine the signs, we can artificially introduce some
  additional randomness. Fix an arbitrary ordering on $\{-1,
  1\}^{V_N}$. For each $\sigma = \{ \sigma_v \}_{v \in V_N} \in \{-1,
  1\}^{V_N}$, define the function $f_\sigma : \R^{V_N} \to \R$ by
  \[ f_\sigma(Z) = \P \( \eta_{N, v} = \sigma_v\sqrt{Z_v} - \sqrt{2t} \text{ for all $v \in V_N$} \mmid \(\eta_{N, v} + \sqrt{2t}\)^2 = Z_v \text{ for all $v \in V_N$}\). \]
  Let $U$ be uniformly distributed on $[0, 1]$ and independent of
  $\eta_N$ and $X_N$. For any $u \in [0, 1]$ and $Z \in \R^{V_N}$, we
  may define
  \[ \sigma^*(u, Z) = \max \left\{ \sigma \in \{-1,1\}^{V_N} : u \ge \sum_{\rho < \sigma} f_\rho(Z) \right\}. \]
  We can then define
  \[ \zeta_{N, v} = \frac{1}{2} \eta_{N, v}^2 + \frac{1}{c_v}\L_{\tau^+(t)}^{X_N}(v) \]
  \[ \eta'_{N, v} = \sigma^*\( U, 2\zeta_{N, v} \) \sqrt{2 \zeta_{N, v}} - \sqrt{2t}. \]

  We are now in the setting of Lemma \ref{lem:domination}, which gives
  \[ \P \( \text{for some $v \in V$, both $\L^{X_N}_{\tau^+(t)}(v) > 0$ and $\eta'_{N, v} + \sqrt{2t} < 0$} \) \le \epsilon, \]
  or equivalently (by Remark \ref{rmk:domination}),
  \[ \P \( \sqrt{\L^{X_N}_{\tau^+(t)}(x)} > \frac{1}{\sqrt{2}} \max \( 0, \eta'_{N, x} + \sqrt{2t} \) \text{ for some $x \in V$} \) \le \epsilon. \]

  Now, let $\eta$ and $X$ be the GFF and a continuous time random walk
  on $G$, respectively. By the relationship between $G_N$ and $G$
  described in Proposition \ref{prop:projection}, we have
  \[ \P \( \left\{ \frac{1}{\sqrt{2}} \max \( 0, \eta_x + \sqrt{2t} \) \right\}_{x \in V} \in A \) = \P \( \left\{ \frac{1}{\sqrt{2}} \max \( 0, \eta'_{N, x} + \sqrt{2t} \) \right\}_{x \in V} \in A \) \]
  \[ \ge \P \( \left\{ \sqrt{\L^{X_N}_{\tau^+(t)}(x)} \right\}_{x \in V} \in A \) - \epsilon = \P \( \left\{ \sqrt{\frac{1}{c_x} \L^X_{\tau^+(t)}(x)} \right\}_{x \in V} \in A \) - \epsilon. \]
  This holds for each $\epsilon > 0$, so taking $\epsilon \rightarrow
  0$, we obtain
  \[ \P \( \left\{ \frac{1}{\sqrt{2}} \max \( 0, \eta_x + \sqrt{2t} \) \right\}_{x \in V} \in A \) \ge \P \( \left\{ \sqrt{ \frac{1}{c_x} \L^X_{\tau^+(t)}(x)} \right\}_{x \in V} \in A \), \]
  which proves the stochastic domination.
\end{proof}

\section{Application to cover times} \label{sec:cover-times}

Theorem \ref{thm:domination} provides good control over the
relationship between local times and the Gaussian free field. By
showing that various quantities are concentrated around their
expectation, one can deduce results pertaining to cover times. In
fact, the exact same arguments used in proving Theorem 1.2 of
\cite{D14} carry through, replacing Theorem 2.3 there with Theorem
\ref{thm:domination} of the previous section. For the sake of
completeness, we repeat the main parts of the argument from
\cite{D14}. It should be mentioned that the argument for the upper
tail bound is originally from \cite{DLP12} (see \S 2.2).

First, we record two auxiliary results used in \cite{D14}. Recall the
notation that $M = \E \max_{x \in V} \eta_x$ for the Gaussian free
field $\eta$ and $R = \max_{x, y \in V} \E \( \eta_x - \eta_y \)^2$.

\begin{lemma}[Lemma 2.1 of \cite{D14}] \label{lem:inverse-local-time-concentration}
  Let $X$ be a continuous time random walk on an electrical network $G
  = (V, E)$. Let $\ctot = \sum_{x, y \in V} c_{xy}$ be the total
  conductance of $G$. For any $t \ge 0$ and $\lambda \ge 1$,

  \[ \P\( \left| \tau^+(t) - \ctot \cdot t \right| \ge \frac{1}{2} \( \sqrt{\lambda R t} + \lambda R \) \ctot \) \le 6 \exp \( - \frac{\lambda}{16} \). \]
\end{lemma}
\begin{proof}
  See Lemma 2.1 of \cite{D14} and the associated remark. We have
  replaced $2|E|$ by $\ctot$.
\end{proof}

\noindent The next result is a well-known Gaussian concentration
bound. See for example Theorem 7.1, Equation (7.4) of \cite{L01}.

\begin{proposition} \label{prop:gaussian-concentration}
  Let $\{ \eta_x : x \in S \}$ be a centered Gaussian process on a
  finite set $S$, and suppose $\E \eta_x^2 \le \sigma^2$ for all $x
  \in S$. Then, for $\alpha > 0$,

  \[ \P \( \left| \max_{x \in S} \eta_x - \E \max_{x \in S} \eta_x \right| \ge \alpha \) \le 2 \exp \( - \frac{ \alpha^2}{2 \sigma^2} \). \]
\end{proposition}

\noindent Note that by symmetry, $\max$ can be replaced by $\min$ in
Proposition \ref{prop:gaussian-concentration}, which is the version
that we will use. We now give a proof of Theorem
\ref{thm:cover-concentration}, closely following the proof of Theorem
1.2 in \cite{D14}.

\begin{proof}[Proof of Theorem \ref{thm:cover-concentration}.]
  We will prove Theorem \ref{thm:cover-concentration} in the slightly
  more general setting where $G = (V, E)$ is an electrical network. As
  before, define $\ctot = \sum_{x, y \in V} c_{xy}$.

  We first estimate $\taucov$ in terms of $\tau^+$. Let $\beta \ge 3$
  be a parameter to be specified later. In what follows, we will often
  use the fact that
  \[ R = \max_{x, y \in V} \E \(\eta_x - \eta_y\)^2 \ge \max_{x \in V} \E \eta_x^2. \]
  To prove an upper bound, let
  $t^+ = \frac{(M + \beta \sqrt{R})^2}{2}$, and define the event
  \[ E = \left\{ \min_{x \in V} \( \L_{\tau^+(t^+)}(x) + \frac{1}{2} \eta^2_x \) \ge \frac{\beta^2 R}{8} \right\}, \]
  where $\eta$ is an independent copy of the Gaussian free field as in
  Theorem \ref{thm:second-ray-knight}. We also have by Proposition
  \ref{prop:gaussian-concentration} that
  \[ \P \( \min_{x \in V} \frac{1}{2} \( \eta_x + \sqrt{2t^+} \)^2 \le \frac{\beta^2 R}{8} \) \le \P \( \min_{x \in V} \( \eta_x + \sqrt{2t^+} \) \le \frac{\beta \sqrt{R}}{2} \) \]
  \[ = \P \( \min_{x \in V} \eta_x \le -M - \frac{\beta \sqrt{R}}{2} \)  \le 2 e^{-\frac{\beta^2}{8}}, \]
  so that in light of the isomorphism theorem (Theorem
  \ref{thm:second-ray-knight}),
  \begin{equation} \label{eq:P(E)}
    \P \( E^c \) \le 2e^{-\frac{\beta^2}{8}}.
  \end{equation}

  \noindent Suppose now that $\taucov > \tau^+(t^+)$. Then, $\L_{\tau^+(t^+)}(x) =
  0$ for some $x \in V$. Since
  \[ \P \( \eta^2_x \ge \frac{\beta^2 R}{4} \) \le 2e^{-\frac{\beta^2}{8}} \]
  and $\eta$ is independent of the random walk, it follows that
  \begin{equation} \label{eq:P(E|tau)}
    \P \( E \mmid \taucov > \tau^+(t^+) \) \le 2e^{-\frac{\beta^2}{8}}.
  \end{equation}
  Combining equations (\ref{eq:P(E)}) and (\ref{eq:P(E|tau)}), we
  conclude that
  \[ \P \( \taucov > \tau^+(t^+) \) \le \frac{2e^{-\frac{\beta^2}{8}}}{1 - 2e^{-\frac{\beta^2}{8}}} \le 6e^{-\frac{\beta^2}{8}}. \]

  \noindent For the lower bound, let $t^- = \frac{(M - \beta
    \sqrt{R})^2}{2}$. By Theorem \ref{thm:domination}, we have
  \[ \P \( \taucov < \tau^+(t^-) \) = \P \( \min_{x \in V} \L_{\tau^+(t^-)}(x) > 0 \) \le \P \( \min_{x \in V} \( \eta_x + \sqrt{2t^-} \) > 0 \) \]
  \[ = \P \( \min_{x \in V} \eta_x > -M + \frac{\beta \sqrt{R}}{2} \) \le 2e^{-\frac{\beta^2}{2}}, \]
  where the last inequality follows again from Proposition
  \ref{prop:gaussian-concentration}.

  Combining the upper and lower bounds, it follows that
  \[ \P \( \tau^+(t^-) \le \taucov \le \tau^+(t^+) \) \ge 1 - 8e^{-\frac{\beta^2}{8}}. \]
  For $\lambda \ge 9$, we now take $\beta = \sqrt{\lambda}$. Note that
  \[ \ctot \cdot t^+ + \frac{1}{2} \( \sqrt{\lambda R t^+} + \lambda R \) \ctot = \frac{\ctot}{2} \( M^2 + 3 \sqrt{\lambda R} M + 3 \lambda R \) \]
  \[ \ctot \cdot t^- - \frac{1}{2} \( \sqrt{\lambda R t^-} + \lambda R \) \ctot = \frac{\ctot}{2} \( M^2 - 3 \sqrt{\lambda R} M - \lambda R \), \]
  so by Lemma \ref{lem:inverse-local-time-concentration},
  \[ \P \( \tau^+(t^+) \ge \frac{\ctot M^2}{2} + \frac{3 \ctot (\sqrt{\lambda R} M + \lambda R)}{2} \) \le 6 \exp \( - \frac{\lambda}{16} \) \]
  \[ \P \( \tau^+(t^-) \le \frac{\ctot M^2}{2} - \frac{3 \ctot (\sqrt{\lambda R} M + \lambda R)}{2} \) \le 6 \exp \( - \frac{\lambda}{16} \). \]
  We thus conclude that for $\lambda \ge 9$,
  \[ \P \( \left| \taucov - \frac{\ctot M^2}{2} \right| \ge \frac{3}{2} \ctot (\sqrt{\lambda R} M + \lambda R) \) \le 20 \exp \( - \frac{\lambda}{16} \). \]
  We obtain Theorem \ref{thm:cover-concentration} upon an appropriate
  rescaling of $\lambda$, noting that $\ctot = 2 |E|$ in the case
  where all conductances are $1$.
\end{proof}

\section{Acknowledgements} \label{sec:acknowledgements}

We are greatly indebted to Jian Ding for suggesting the problem and
valuable discussions. We also thank Amir Dembo, James Lee, and Yuval
Peres for very helpful feedback and advice at various stages.

\section{Appendix}

\subsection{Proof of Lemma \ref{lem:brownian-disk-avoidance}}

To break up the proof, we first establish a lemma.

\begin{lemma} \label{lem:brownian-disk-avoidance-helper}
  Let $r > 0$ be given, and consider any point $y \in \R^2$ such that
  $|y| > r$. Let $\{ W^y_t \}_{t \ge 0}$ be a standard planar Brownian
  motion started at $y$. Then,

  \[ \P \( \inf_{t \in [0, 1]} |W^y_t| \le r \) \le \inf_{0 \le \alpha \le |y|} 2 \( \frac{\log \alpha^{-1}}{\log r^{-1}} + \alpha^2 \). \]
\end{lemma}
\begin{proof}
  Note that $\P \( \inf_{t \in [0, 1]} |W^y_t| \le r \)$ is decreasing
  in $|y|$, so it suffices to show the inequality only for $\alpha =
  |y|$. Let $s = \frac{1}{|y|}$. Define two stopping times

  \[ T = \inf \{ t \ge 0 : |W^y_t| \not\in [r, s] \} \]
  \[ T' = \inf \{ t \ge 0 : |W^y_t| \not\in [r, \infty) \} \]

  Now, consider the stopped martingale $X_t = \log |W^y_{T \wedge t}|$,
  noting that $X_0 = \log |y|$ and $X_t \in [\log r, \log s]$. By the martingale
  property, we have

  \[ \P \( X_1 = \log r \) \le \frac{\log s - \log |y|}{\log s - \log r} = \frac{2 \log |y|^{-1}}{\log |y|^{-1} + \log r^{-1}} \le \frac{2 \log |y|^{-1}}{\log r^{-1}}. \]
  Moreover, by Doob's maximal inequality\footnote{We use Doob's
    maximal inequality for brevity only. Other methods such as the
    reflection principle would serve just as well; the bound on $\sup
    |W^y_t|$ does not need to be sharp for our purposes.} on the
  submartingale $|W^y_t|^2$,
  \[ \P \Big( \min(T, 1) \ne \min(T', 1) \Big) \le \P \( \sup_{t \in [0, 1]} |W^y_t| \ge s \) \le \frac{|y|^2 + 1}{s^2} \le 2 |y|^2. \]
  It follows that
  \[ \P \( \inf_{t \in [0, 1]} |W^y_t| \le r \) = \P \Big( T' \le 1 \Big) \le \P \Big( X_1 = \log r \text{ or } \min(T, 1) \ne \min(T', 1) \Big) \]
  \[ \le 2 \( \frac{\log |y|^{-1}}{\log r^{-1}} + |y|^2 \), \]
  as desired.
\end{proof}

\begin{proof}[Proof of Lemma \ref{lem:brownian-disk-avoidance}]
  Define $\lambda' = \lambda^{\frac{1}{\log \epsilon^{-1}}}$. Recall
  that the probability density of the standard two-dimensional
  Gaussian is bounded above by $\frac{1}{2 \pi}$, and so the
  probability density of $W_\epsilon$ is bounded above by $\frac{1}{2
    \pi \epsilon}$. Thus,
  \[ \P \( |W_\epsilon|^2 \le \lambda' \) \le \frac{1}{2 \pi \epsilon} \cdot \pi \lambda' = \frac{1}{2 \epsilon} \exp \( - \frac{\log \lambda^{-1}}{\log
    \epsilon^{-1}} \). \]

  We now apply Lemma \ref{lem:brownian-disk-avoidance-helper} with $y =
  W_\epsilon$, $r = \sqrt{\lambda}$, and taking $\alpha =
  \sqrt{\lambda'}$ in the infimum. This gives
  \[ \P \( \inf_{\epsilon \le t \le 1} |W_t|^2 < \lambda \mmid |W_\epsilon|^2 \ge \lambda' \) \le 2 \( \frac{\log \lambda'^{-1}}{\log \lambda^{-1}} + \lambda' \) \]
  \[ = \frac{2}{\log \epsilon^{-1}} + 2 \exp \( - \frac{\log \lambda^{-1}}{\log
    \epsilon^{-1}} \) \le \frac{2}{\log \epsilon^{-1}} + \frac{5}{2
    \epsilon} \exp \( - \frac{\log \lambda^{-1}}{\log \epsilon^{-1}} \). \]
  This along with the previous inequality proves the corollary.
\end{proof}

\subsection{Proof of Lemma \ref{lem:conditioned-path}}

\begin{proof}
  Define
  \[ f(x) = \left\{
  \begin{array}{lr}
    x & : 0 \le x \le N \\
    N + \frac{1}{r} & : x = N + 1
  \end{array}
  \right. \]

  Note that $f(X)$ is a martingale. Thus, for a walk started at $k$,
  the probability of hitting $0$ before $N + 1$ is
  \[ \frac{f(N + 1) - f(k)}{f(N + 1) - f(0)} = \frac{N - k + \frac{1}{r}}{N + \frac{1}{r}}.\]
  It follows that for $1 \le k < N$,
  \[ \frac{\P \Big( X_{t + 1} = k + 1 \,\Big|\, X_t = k, X_\tau = 0 \Big)}{\P \Big( X_{t + 1} = k - 1 \,\Big|\, X_t = k, X_\tau = 0 \Big)} = \frac{N - k - 1 + \frac{1}{r}}{N - k + 1 + \frac{1}{r}} = \frac{c'_{k, k + 1}}{c'_{k - 1, k}}, \]
  where
  \[ c'_{k, k + 1} = \frac{\(N - k - 1 + \frac{1}{r}\)\(N - k + \frac{1}{r}\)}{\frac{1}{r}\(1 + \frac{1}{r}\)}. \]

  Thus, the transition probabilities of $X$ conditioned on $X_\tau =
  0$ are exactly the unconditioned transition probabilities of
  $Y$. Consequently, their paths have the same distribution.
\end{proof}

\subsection{Proof of Lemma \ref{lem:exp-conc}}

\begin{proof}
  For any $t < \frac{1}{\mu}$, we have by direct calculation
  \[ \log \( \E \exp \( t \sum_{i = 1}^N X_i \) \) = N \log \( \frac{1}{1 - \mu t} \). \]
  \noindent If in fact $|t| \le \frac{\alpha}{2 \mu}$, we have
  \[ \log \( \frac{1}{1 - \mu t} \) = \sum_{k = 1}^\infty \frac{\mu^k t^k}{k} \le \mu t + \mu^2 t^2 = \mu t (1 + 2 \mu t) - \mu^2 t^2. \]
  \noindent and so
  \[ \frac{\E \exp \( t \sum_{i = 1}^N X_i \)}{\exp \( t(1 + 2 \mu t)\mu N \) } \le e^{-\mu^2 t^2 N}. \]
  \noindent By Markov's inequality with $t = \frac{\alpha}{2 \mu}$ and $t = - \frac{\alpha}{2 \mu}$, we obtain
  \[ \P\( (1 - \alpha) \mu N \le \sum_{i = 1}^N X_i \le (1 + \alpha) \mu N \) \le 2e^{-\frac{1}{4} \alpha^2 N}. \]
\end{proof}

\end{document}